\documentclass[oneside, 10pt]{amsart}

\usepackage{amsmath, amsfonts, amssymb, amsthm}

\newtheorem{theorem}{Theorem}[section]
\newtheorem{lemma}[theorem]{Lemma}

\newtheorem{proposition}[theorem]{Proposition}

\theoremstyle{definition}

\newtheorem{remark}[theorem]{Remark}

\renewcommand{\proofname}{Proof}

\numberwithin{equation}{section}

\begin{document}

\title{Sharp inequalities over the unit polydisc}

\author{Marijan Markovi\'{c}}

\address{Faculty of Natural Sciences and Mathematics\endgraf
University of Montenegro\endgraf
Cetinjski put b.b.\endgraf
81000 Podgorica\endgraf
Montenegro}

\email{marijanmmarkovic@gmail.com}



\dedicatory{Dedicated to Professor Miodrag Mateljevi\'{c}\\  on the occasion of his 65th
birthday}

\begin{abstract}
Motivated by some                 results due to Burbea we prove that if a certain sharp
integral inequality  holds for  functions in the unit polydisc which  belong to concrete
Hardy  spaces, then it also holds, in an appropriate form, in the case of functions from
arbitrary Hardy spaces. We also examine     the equality case. We present an application
of  this main result  to  a  Burbea inequality  which    includes an  isoperimetric type
inequality  as a special  case.
\end{abstract}

\maketitle

\section{Introduction and the main theorem}
\subsection{Introduction} In        this paper we are  interested in a   certain kind of
integral inequalities for analytic functions in Hardy spaces in  the unit polydisc.  One
of such inequalities is contained in a theorem  due to Burbea which is formulated  below.
Before  formulation we  recall what the  generalized Hardy spaces stand   for. We recall
the  definition   of classical  Hardy  spaces in the unit  polydisc in the next  section.

Introduce       firstly the basic notations  which will be used. Let $\mathbf{C}$ be the
complex  plane.  Denote by $\mathbf{U}$ the open  unit disc, i.e.,       the set $\{z\in
\mathbf{C}:|z|<1\}$. For an integer   $n\ge 1$ let $\mathbf{C}^n$       stand    for the
$n$-dimensional complex vector space.                 The direct product $\mathbf{U}^n =
\underbrace{\mathbf{U}\times \dots \times \mathbf{U}}_n$   is the unit    polydisc,  and
$\mathbf{T}^n = \underbrace{\mathbf{T}\times \dots \times   \mathbf{T}}_n$ is  the  unit
torus  in $\mathbf{C}^n$.

Let $\mathbf{Z}_+$  be the set of all non--negative integers. Denote by $\mathbf{Z}^n_+=
\underbrace{\mathbf{Z}_+\times \dots \times \mathbf{Z}_+}_n $         the  set   of  all
multi--indexes. For    any complex number   $q$   the shifted factorial  (the Pochhammer
symbol) is
\begin{equation*}
(q)_\beta= \left\{
\begin{array}{ll}
q(q + 1)\cdots (q + \beta  - 1),   & \hbox{if $\beta\ge1$,} \\
1,                                 & \hbox{if $\beta=0$},
\end{array}
\right.
\end{equation*}
where     $\beta$ is an integer in $\mathbf{Z}_+$. One  may extend   this  definition as
follows. For every $q=(q_1,\dots,q_j,\dots,q_n)\in\mathbf{C}^n$          and    $\alpha=
(\alpha_1,\dots,\alpha_j,\dots,\alpha_n)\in\mathbf{Z}^n_+$                        denote
\begin{equation*}
(q)_\alpha = \prod_{j=1}^n (q_j)_{\alpha_j}.
\end{equation*}

Denote by $\mathbf{R}^n_+$   the  set $\{(x_1,\dots,x_j,\dots, x_m)\in \mathbf{R}^n: x_j
\ge0,\, j=1,\dots,n\}$.

For $q = (q_1,\dots,q_n)\in\mathbf{R}^n_+\backslash\{0\}$ the generalized Hardy space in
the unit    polydisc, denoted by      $H_q(\mathbf{U}^n)$, is the  space of all analytic
functions  $f$  in   $\mathbf{U}^n$ for which the following     norm  is          finite
\begin{equation*}
\|f\|_q^2\ =\sum_{\alpha\in\mathbf{Z}^n_+}\frac{\alpha!}{(q)_\alpha} |a_\alpha|^2,
\end{equation*}
where $a_\alpha =a_\alpha(f),\,  \alpha\in\mathbf{Z}_+^n$ is the    $\alpha$-coefficient
in the Taylor expansion for $f$,                                           i.e., $f(z) =
\sum _{\alpha\in\mathbf{Z}_+^n}a_\alpha z^\alpha$.  The   space $H_q(\mathbf{U}^n)$ is a
reproducing kernel                                         Hilbert space with the kernel
\begin{equation*}
K_q(z,w) = \prod_{j=1}^n \frac 1 {(1-z_j\overline{w}_j)^{q_j}},
\end{equation*}
$z=(z_1,\dots,z_j,\dots,z_n)\in \mathbf{U}^n,\,          w= (w_1,\dots,w_j,\dots,w_n)\in
\mathbf{U}^n$. The details  of the  construction  of    generalized  Hardy spaces, based
on some  facts from theory of reproducing kernels, may be found            in the Burbea
paper~\cite{BURBEA.ILLINOIS.83} in the case of the unit disc;       the case of the unit
polydisc    is given in~\cite{BURBEA.ILLINOIS.87}.     For    the  theory of reproducing
kernels we refer to the                          work of Aronszajn~\cite{ARONSZAJN.TAMS}.

The theorem which  follows serves as a  motivation for our main    result stated  in the
third subsection.

\begin{theorem}[cf.~\cite{BURBEA.ILLINOIS.87}]\label{TH.GENERAL.HARDY.POLYDISC}
Let $q_j\in \mathbf{R}^n_+\backslash\{0\}$ and $f_j(z)\in H_{q_j}(\mathbf{U}^n)$ for all
$j=1,\, 2,\dots,m$, where  $m\ge  2$ is an integer. Denote $q = \sum_{j=1}^m q_j$.  Then
\begin{equation*}
\prod_{j=1}^m f_j\in H_{q}(\mathbf{U}^n)
\end{equation*}
and
\begin{equation*}
\left\|\prod_{j=1}^m f_j\right\|_{q}\ \ \le\ \ \prod_{j=1}^n \|f_j\|_{q_j}.
\end{equation*}
The        equality sign attains if and only if either $\prod_{j=1}^mf_j\equiv0$ or each
$f_j(z)\,  (j=1,\, 2,\dots,m)$                                            is of the form
\begin{equation*}
f_j(z)=C_jK_{q_j}^w(z)
\end{equation*}
for        some (common)                    $w\in\mathbf{U}^n$ and a constant $C_j\ne 0$.
\end{theorem}

We  have used the following notation to express the      extremal functions. If $F(z,w)$
is any   function   of two variables and if $w$ is fixed,    then $F_w$ and $F^w$ denote
the following   restricted function  $F_{w}(z) =F^{w}(z)  = F(z,w)$. Similar meaning has,
for example,       $F_{zw}(\omega)$ if $F(z,w,\omega)$ is a function of three  variables.

\begin{remark}
The above formulated theorem   is Theorem 4.1           in ~\cite{BURBEA.ILLINOIS.87}. A
generalization of it may be  found in the same paper for   Reinhardt domains. A    proof
in the case of Reinhardt domains may be obtained      modifying the arguments      given
in~\cite{BURBEA.TAMS} for the case  of the unit ball in $\mathbf{C}^n$.    The principal
idea is to  use Carleman's  approach~\cite{CARLEMAN.MATH.Z}     in proving the classical
isoperimetric  inequality for minimal surfaces.
\end{remark}

Let  $dA$ stand  for the area  measure in $\mathbf{C}$. For a real $q>1$ introduce   the
normalized  weighted measure in the unit disc
\begin{equation*}
dA_{q-2} (z) = \frac {q-1}\pi (1-|z|^2)^{q-2}\,dA(z).
\end{equation*}
An  element          $(q,\dots,q)\in\mathbf{R}^n$ will be abbreviated  as  $\mathbf {q}$.
$dA_{\mathbf{q}-\mathbf{2}}$ is the measure on the  unit polydisc $\mathbf{U}^n$   given
by  the product  $\underbrace{dA_{q-2}\times\dots\times dA_{q-2}}_n$.    For $q=1$ it is
convenient to set $dA_{\mathbf{q}-\mathbf{2}} = dA_{-\mathbf{1}}= dm_n$, where $dm_n$ is
the Haar        measure on the unit torus $\mathbf{T}^n$.

For $q \ge 1$    it is not hard to verify                   that  the square of the norm
$\|\cdot\|_{\mathbf {q}}$                               has the integral  representation
\begin{equation}\label{EQ.NORM.GEN.HARDY}
\|f\|^2_{\mathbf {q}} \   = \int |f |^2\, dA_{\mathbf {q} -\mathbf{2}},\quad        f\in
H_{\mathbf{q}}(\mathbf{U}^n).
\end{equation}
In~\eqref{EQ.NORM.GEN.HARDY} we assume integration over $\mathbf{U}^n$   if $q > 1$, and
over  $\mathbf{T}^n$  if $q=1$. In the last case,    the object of    integration is the
radial boundary function for $f$, since the  generalized                     Hardy space
$H_{\mathbf{1}}(\mathbf{U}^n)$ coincides with  the          classical   Hardy      space
$H^2(\mathbf{U}^n)$. The space $H_{\mathbf{q}}(\mathbf{U}^n),\, q>1$            is   the
weighted Bergman space, usually denoted by $L^2_{a,q-2}(\mathbf{U}^n)$.  For $q\in(0,1)$
the     space $H_{\mathbf{q}}(\mathbf{U}^n)$  is  also known   as  the  Bergman--Selberg
space in  the unit polydisc.

One             of         our  aims in  this paper is to establish a theorem similar to
Theorem~\ref{TH.GENERAL.HARDY.POLYDISC}       for analytic functions which belong to the
classical Hardy spaces in the unit polydisc  which do not necessary have   the   Hilbert
structure.  This is done in the last section using the method which will be  established
in the next two sections.

\subsection{Spaces   of analytic functions in the unit polydisc}
Following the          Rudin monograph \cite{RUDIN.BOOK.POLYDISC} we  collect  here  the
definitions and some facts    concerning  the  well known spaces   of analytic functions
in the unit polydisc.

Let us     first say that,  if $(X,\nu)$  is a measure space, then $L^p(X,\nu),\, 0<p\le
\infty$ denotes the  Lebesgue space over  $(X,\nu)$. We write $\|\varphi\|_{L^p(X,\nu)}$
for the norm of $f\in L^p(X,\nu)$.

The     Nevanlinna class $N(\mathbf{U}^n)$ contains all analytic functions $f(z)$ in the
unit                                        polydisc which satisfy the growth  condition
\begin{equation*}
\sup_{0\le r< 1}\int_{\mathbf{T}^n}\log^+ |f(r\zeta )|\, dm_n(\zeta)<\infty.
\end{equation*}
Recall that
\begin{equation*}
\log^+  x  =
 \left\{
\begin{array}{ll}
\log x, & \hbox{if $x> 1$;}\\
0,   & \hbox{if $0<x\le 1$}.
\end{array}\right.
\end{equation*}
In other words, for           the family of functions $\{\mathbf{T}^n\ni\zeta\rightarrow
\log^+|f_r(\zeta)|:0\le r<1\}$,  where  $f_r(z) = f(rz),\, z\in\overline {\mathbf{U}}^n$
is the $r$-dilatation of $f(z)$, is required to lie  in a bounded subset of the Lebesgue
space $L^1(\mathbf{T}^n,m_n)$.  The class $N\sp\ast(\mathbf{U}^n)$ is   the class of all
$f\in N(\mathbf{U}^n)$ for  which  the preceding  family of   functions form a uniformly
integrable family.

We    call a function $\phi(t)$ strongly convex if it is convex on $(-\infty,+\infty),\,
\phi\ge 0,\, \phi$           is non--decreasing, and ${\phi(t)}/t\rightarrow +\infty$ as
$t\rightarrow +\infty$.                 If $\phi$ is a strongly convex function,  define
$H_\phi(\mathbf{U}^n)$ to be the class of all  analytic functions  $f$ in $\mathbf{U}^n$
for  which
\begin{equation*}
\sup_{0\le r<1}\int_{\mathbf{T}^n}  \phi(\log|f(r\zeta)|)\, dm_n(\zeta) <\infty.
\end{equation*}
It happens that the  space  $N\sp\ast(\mathbf{U}^n)$               is the union   of all
$H_\phi(\mathbf{U}^n)$                      (this is the content of        Theorem 3.1.2
in~\cite{RUDIN.BOOK.POLYDISC}).

If $f(z)$  is any function in $\mathbf{U}^n$,  we define its radial boundary    function
$f\sp\ast(\zeta)$ by $f\sp\ast (\zeta)=\lim_{r\rightarrow 1^-} f(r\zeta)$ at every point
$\zeta\in\mathbf{T}^n$ where  the   radial limit exists.   For $f(z)\in N(\mathbf{U}^n)$
it is known that $f\sp\ast(\zeta)$ exists for almost every      $\zeta \in \mathbf{T}^n$.
Moreover,     $\log |f\sp\ast(\zeta)|\in L^1 (\mathbf{T}^n,m_n)$.                 Within
$N\sp\ast(\mathbf{U})$, the $H_\phi$-classes       are characterized   by their boundary
values. Suppose that    $f(z)\in N\sp\ast(\mathbf{U}^n)$ and $\phi$ is  strongly  convex,
then
\begin{equation*}
f(z)\in H_\phi(\mathbf{U}^n)\ \text{if and only if}\ \phi(\log |f\sp\ast(\zeta)|)
\in L^1(\mathbf{T}^n,m_n).
\end{equation*}
If this is the case, then
\begin{equation*}
r\rightarrow \int_{\mathbf{T}^n} \phi(\log|f_r(\zeta)|)\, dm_n(\zeta)
\end{equation*}
is   increasing in $0\le r < 1$ (since $\phi(\log|f_r(\zeta)|)$ is $n$-subharmonic, i.e.,
subharmonic in each variable separately),  and
\begin{equation*}
\lim_{r\rightarrow 1^-}\int_{\mathbf{T}^n}\phi(\log|f_r(\zeta)|)\,    dm_n(\zeta)=
\int_{\mathbf{T}^n} \phi(\log |f\sp\ast(\zeta)|)\, dm_n(\zeta).
\end{equation*}

The           Hardy space in the  unit polydisc $H^p(\mathbf{U}^n)\, (0 <p <\infty)$  is
$H_\phi(\mathbf{U}^n)$ with $\phi(t)=e^{pt}$.                  We write $H^p$ instead of
$H^p(\mathbf{U})$   for the Hardy space in the unit disc. For the theory of Hardy spaces
in the unit disc we refer to Duren's book~\cite{DUREN.BOOK.HP}.   One introduces  a norm
in $H^p(\mathbf{U}^n)$ by
\begin{equation*}
\|f\|_p \, =  \sup_{0\le r<1} M_p(f,r),
\end{equation*}
where we have denoted
\begin{equation*}
M_p(f,r)=\left\{ \int_{\mathbf{T}^n}  |f(r\zeta)|^p  dm_n(\zeta)\right\}^{1/p}.
\end{equation*}
Since     $|f(z)|^p$ is $n$-subharmonic in the expression for $\|f\|_p^p=\sup_{0\le r<1}
M^p_p(f,r)$ we may replace $\sup_{0\le r<1}$   by $\lim_{r\rightarrow 1^-}$.   Therefore,
for $f\in H^p$  we may write                 $\|f\|_p =\|f^*\|_{L^p(\mathbf{T}^n, m_n)}$.

For                 $f\in H^p(\mathbf{U}^n)\, (0<p <\infty)$ we have convergence in mean
\begin{equation*}
\lim_{r\rightarrow 1^-}
\int_{\mathbf{T}^n}|f_r (\zeta)- f\sp\ast(\zeta)|^p\, dm_n(\zeta)= 0.
\end{equation*}
As a consequence, one derives that  every $f(z)\in H^p(\mathbf{U}^n)\,  (1\le p<\infty)$
may be represented as the Poisson integral as well as the Cauchy integral of its  radial
boundary function  $f\sp\ast(\zeta)$.  For example,
\begin{equation*}
f(z)\, =
\int_{\mathbf{T}^n} K(z,\zeta)\, f\sp\ast(\zeta)\, dm_n(\zeta),\quad z\in
\mathbf{U}^n.
\end{equation*}
In the preceding  relation $K(z,\zeta)$ stands for the Cauchy--Szeg\"{o}  kernel for the
unit polydisc given by
\begin{equation*}
K(z,\zeta) = \prod_{j=1}^n\frac1{1-z_j\overline{\zeta}_j},
\end{equation*}
where                         $z=(z_1,\dots,z_j,\dots,z_n )\in\mathbf{U}^n$ and $\zeta =
(\zeta_1, \dots,\zeta_j,\dots,z_n)\in\mathbf{T}^n$.

\subsection{The main result}
In    the   sequel  a weighted measure         in the unit disc is a measure of the form
\begin{equation*}
d\mu(z) =  g(z)\, dA(z),\quad   g(z)>0,\,  z\in \mathbf{U}.
\end{equation*}
A weighted measure in the polydisc    $\mathbf{U}^n$  is a product of $n$ (not necessary
equal)  weighted measures in the unit disc.

The letter $m$ always denotes an  integer $\ge 1$, and  the letter $p$ (with or  without
an index)      any positive number. Let $\Phi:\mathbf{R}^m_+\rightarrow \mathbf{R}_+$ be
continuous and strictly   increasing    in each variable separately,     which  moreover
satisfies $\Phi(x_1,\dots, x_j,\dots,x_m)=0$ if $x_j=0$ for    some $j,\, 1\le j  \le m$.

For     a fixed  weighted measure in the unit disc $\mu$ we will consider  the following
weighted   measure in the unit polydisc $\nu_n = \underbrace{\mu\times\dots\times\mu}_n$.

We    will prove our main  result under an assumption that $\Phi$ and $\mu$ satisfy  the
condition:

\medskip

\noindent($\dag$)  {\it There exist $\tilde{p}_j,\, 0<\tilde{p}_j<\infty,\, j=1,\dots,m$
such that
\begin{equation*}
\Phi(|f_1|^{\tilde{p}_1}\dots,|f_m |^{\tilde{p}_m})\in L^1(\mathbf{U}^n,\nu_n)
\end{equation*}
and
\begin{equation}\label{INEQ.MAIN.THILDA}
\int_{\mathbf {U}^n}
\Phi(|f_1(z)|^{\tilde{p}_1},\dots,|f_m(z)|^{\tilde{p}_m})\, d\nu_n(z)
\le \Phi( \|f_1\|_{\tilde{p}_1}^{\tilde{p}_1},\dots,\|f_m\|_{\tilde{p}_m}^{\tilde{p}_m})
\end{equation}
for all $f_j(z)\in  H^{\tilde{p}_j}(\mathbf{U}^n)\, (j=1,\dots,m)$,    with the equality
sign if and only if either
\begin{enumerate}
\item $\prod_{j=1}^m f_j\equiv 0$   (i.e. $f_j\equiv 0$  for some $j,\, 1\le j\le m$) or
\item each $f_j\, (j=1,\dots,m)$ is equal to $f_j(z)={\Psi_j^n}(z)\not\equiv 0$,   where
$(\dots,\Psi_j^n,\dots)$         belongs to a class denoted by $\mathcal{E}(\Phi,\nu_n)$.
\end{enumerate}
}

\medskip

In       what follows  the  class $\mathcal{E}(\Phi,\nu_n)$ we will  call  the family of
extremals for the  inequality~\eqref{INEQ.MAIN.THILDA}.

\begin{remark}
Note  that if $C_j,\, j=1,\dots,m$ are  constants which satisfy   $|C_j|=1$, and      if
we    take some  $(\dots,\Psi_j^n,\dots)\in \mathcal{E}(\Phi,\nu_n)$,     then we   also
have  $(\dots,C_j\Psi_j^n,\dots)\in\mathcal{E}(\Phi,\nu_n)$.        If $\Phi$  satisfies
$\Phi(\dots,\alpha_jx_j,\dots)    = \alpha_j  \Phi(\dots,x_j,\dots),\, \alpha_j>0,\, x_j
\ge 0,\, j=1,\dots,m$, then the previous conclusion holds for all $C_j>0,\, j=j,\dots,m$.
\end{remark}

Our  main  goal in this paper  is to prove

\begin{theorem}\label{TH.MAIN}
Assume that $\Phi$ and $\mu$      satisfy the condition $(\dag)$. Let $f_j(z)\in H^{p_j}
(\mathbf U^n)$ $(0<p_j<\infty)$ for all $j=1,\dots,m$. Then
\begin{equation*}
\Phi(|f_1 |^{p_1},\dots,|f_m|^{p_m})\in L^1(\mathbf{U}^n,\nu_n)
\end{equation*}
with
\begin{equation}\label{INEQ.MAIN}
\int_{\mathbf {U}^n}         \Phi (|f_1(z)|^{p_1},\dots,|f_m(z)|^{p_m} )\, d\nu_n(z) \le
\Phi(\|f_1\|_{p_1}^{p_1},\dots,\|f_m\|_{p_m}^{p_m}).
\end{equation}
Moreover,
\begin{enumerate}
\item            each extremal $\Psi_j^n,\, j= 1,\dots,m$                        for the
inequality~\eqref{INEQ.MAIN.THILDA}, i.e.,                    $(\dots,\Psi_j^n,\dots)\in
\mathcal{E}(\Phi,\nu_n)$,                       vanishes nowhere in  $\mathbf{U}^n$, and
\item equality  attains in~\eqref{INEQ.MAIN} if and only if either $\prod_{j=1}^m    f_j
\equiv 0$ or  each  $f_j\, (j=1,\dots,m)$          is of  the form             $f_j(z) =
\Psi_{j}^n (z)^{\tilde{p}_j/p_j}$ for some                    $(\dots,\Psi_j^n,\dots)\in
\mathcal{E}(\Phi,\nu_n)$.
\end{enumerate}
\end{theorem}

\begin{remark}
Note   that  if we take  $\Phi (x_1,\dots,x_m) = x_1\cdots x_m$   and  $d\mu = dA_{m-2}$,
then we have $d\nu_n =dA_{\mathbf{m}-2}$, and the condition $(\dag)$ is satisfied in the
Hilbert case, i.e., when  $\tilde{p}_1  = \dots=\tilde{p}_m=2$,  what  states the Burbea
theorem mentioned  in the Introduction.
\end{remark}

\section{Preliminaries for the proof of the main theorem}
\subsection{The one--dimensional case}
The       case $n=1$   of our  Theorem~\ref{TH.MAIN} is straightforward to obtain and we
consider it here. We prove

\begin{theorem}\label{TH.MAIN.1}
Assume that   $\Phi$ and $\mu$ satisfy the condition $(\dag)$. Let $f_j(z)\in  H^{p_j}\,
(0<p_j< \infty)$  for all $j=1,\dots,m$. Then
\begin{equation*}
\Phi(|f_1 |^{p_1},\dots,|f_m |^{p_m})\in L^1(\mathbf{U},\mu)
\end{equation*}
with
\begin{equation}\label{INEQ.MAIN.1}
\int_{\mathbf {U}}  \Phi (|f_1(z)|^{p_1},\dots,|f_m(z)|^{p_m} )\, d\mu(z)\le
 \Phi(\|f_1\|_{p_1}^{p_1},\dots,\|f_m\|_{p_m}^{p_m}).
\end{equation}
Moreover:
\begin{enumerate}
\item Every $\Psi_j^1,\, j=1,\dots,m,\, (\dots,\Psi_j^1,\dots,)\in\mathcal{E}(\Phi,\mu)$,
annihilates nowhere in the unit disc.
\item Equality  attains  in~\eqref{INEQ.MAIN.1} if and only if either $\prod_{j=1}^m f_j
\equiv0$ or each $f_j\, (j=1,\dots,m)$     is of the form                       $f_j(z)=
\Psi_{j}^1(z)^{\tilde{p}_j/p_j}$              for      some   $(\dots,\Psi_j^1,\dots)\in
\mathcal{E}(\Phi,\mu)$.
\end{enumerate}
\end{theorem}

\begin{proof}
Without lost of generality, suppose that $f_j\not\equiv0$ for all $j=1,\dots,m$. By the
Riesz theorem it  is possible to obtain the       factorization $f_j(z) = B_j(z)h_j(z)$,
where $B_j$ is the Blaschke product for $f_j$.    Recall that  one  takes $B_j\equiv 1$,
if $f_j$ is zero--free. Since  $h_j$   does not vanish in the unit disc, it is possible
to obtain a  branch  $\tilde{h}_j(z)     = h_j(z)^{p_j/\tilde{p}_j}$ there.       Since
$|B_j(z)|\le   1$                                   everywhere in  the disc,    we have
\begin{equation}\label{EQ.FJ.HJ}
|f_j(z)|\le |h_j(z)|,\quad  z\in\mathbf{U}.
\end{equation}
Since        $|B_j(\zeta)|=1$    for almost every $\zeta\in\mathbf{T}$, it follows that
$|h_j(\zeta)| = |f_j(\zeta)|$  a.e. on $\mathbf{T}$. Thus
\begin{equation}\label{EQ.TILDE.HJ}
\|\tilde{h}_j\|^{\tilde{p}_j}_{\tilde{p}_j}  =  \|f_j\|_{p_j}^{p_j},
\end{equation}
which implies $\tilde{h}_j\in H^{\tilde{p}_j}$.

In      view of~\eqref{EQ.FJ.HJ}, since $\Phi$ is increasing in each variable, we  have
\begin{equation}\label{EQ.INT.PHI.FJ}
\int_\mathbf{U} \Phi(\dots, |f_j(z)|^{p_j},\dots) \, d\mu(z)
\le\int_{\mathbf U} \Phi(\dots,|\tilde{h}_j(z)|^{\tilde{p}_j},\dots)\, d\mu(z).
\end{equation}
Regarding                                      the condition  $(\dag)$, we first obtain
$\Phi(\dots,|\tilde{h}_j|^{\tilde{p}_j}|,\dots)\in L^1(\mathbf{U},\mu)$.     This means
that                          both integrals in~\eqref{EQ.INT.PHI.FJ} are finite, hence
$\Phi(\dots,|f_j|^{{p}_j},\dots)\in L^1(\mathbf{U},\mu)$.                       Further,
\begin{equation}\begin{split}\label{EQ.INT.UN.PHI.TILDE.HJ}
\int_{\mathbf U}   \Phi(\dots,|\tilde{h}_j(z)|^{\tilde{p}_j},\dots )\, d\mu(z)\le
\Phi(\dots,\|\tilde{h}_j\|^{p_j}_{p_j},\dots).
\end{split}\end{equation}
Regarding~\eqref{EQ.TILDE.HJ},          the inequality of this theorem follows from the
relations~\eqref{EQ.INT.PHI.FJ}  and~\eqref{EQ.INT.UN.PHI.TILDE.HJ}.

Let                us consider now the second half of this theorem. If equality attains
in~\eqref{INEQ.MAIN.1},        then equality must hold in~\eqref{EQ.INT.UN.PHI.TILDE.HJ}
and~\eqref{EQ.INT.PHI.FJ}.  Applying  the equality statement of ($\dag$), we infer that
equality    holds   in~\eqref{EQ.INT.UN.PHI.TILDE.HJ} if and only if $\tilde{h}_j(z)  =
\tilde{\Psi}_j^1(z),\, j=1,\dots,m$        for  some $(\dots,\tilde{\Psi}_j^1,\dots)\in
\mathcal{E}(\Phi,\mu)$. It  follows  that each    $\tilde{\Psi}_j^1$ is zero--free. Now,
equality holds in~\eqref{EQ.INT.PHI.FJ} if and only if $|B_j(z)|\equiv 1$       for all
$j=1,\dots,m$. This means that
\begin{equation*}
f_j(z) =  C_j \tilde{\Psi}_j^1(z)^{\tilde{p}_j/p_j} =
 \{ \tilde{C_j} \tilde{\Psi}_j^1(z)\}^{\tilde{p}_j/p_j} ,
\end{equation*}
where               $|C_j|=|\tilde{C}_j| = 1$ are     constants (for all $j=1,\dots,m$).
\end{proof}

\begin{remark}
Due  to the non--existence of a direct analogue of the  Riesz factorization theorem for
Hardy   spaces (and for the Nevanlinna space, as well)     in    the unit polydisc, one
cannot   prove our main theorem in  a such easy way  as   in the case $n=1$.    However,
it is possible  to          give a proof  using a  factorization theorem, but with some
constraints.   This  will  be shown at the end of  the next  section where we prove our
main result.
\end{remark}

In our complete    proof of Theorem~\ref{TH.MAIN}    the main role play logarithmically
subharmonic  functions   in the Hardy classes (the    classes $PL_p$ for a positive $p$,
which will be  introduced in the  sequel). Recall, a    function $U$ is logarithmically
subharmonic in a domain    $D$  if $U\equiv 0$, or   if it is  possible to represent it
in the form  $U(z) = e^{u(z)},\, z\in D$, where $u(z)$ is a subharmonic function in $D$.
In the next subsection  we will  need the following           two lemmas concerning the
(logarithmically) subharmonic functions. Their  proofs may be found at    the beginning
of the Ronkin monograph~\cite{RONKIN.BOOK}. Actually,    regarding the following remark,
it  is enough to consider both lemmas for subharmonic  functions.   A function $U(x,y)$
is         logarithmically subharmonic in  $D$  if and only if  $e^{\alpha x  +\beta y}
U(x,y)$     is subharmonic in $D$ for every choice of real numbers $\alpha$ and $\beta$.

\begin{lemma}[cf.~\cite{RONKIN.BOOK}]\label{LE.RONKIN.1}
Let $X$ be a non--empty set and let $\{U_\alpha,\, \alpha\in X\}$        be a family of
(logarithmically) subharmonic functions                           in a domain $D$. Then
\begin{equation*}
u(z)=\sup_{\alpha \in X} U_\alpha(z)
\end{equation*}
is also (logarithmically)   subharmonic in $D$ if it is upper semi--continuous  in this
domain.
\end{lemma}

\begin{lemma}[cf.~\cite{RONKIN.BOOK}]\label{LE.RONKIN.2}
Let $U(z,x)$ be upper semi--continuous  in  $D\times X$,  where  $D$ is  a domain   and
$X$    is a topological space.       Moreover,  let  $\nu$ be a finite  measure on  $X$.
Then
\begin{equation*}
U(z) = \int_X  U(z,x)\,  d\nu(x)
\end{equation*}
is (logarithmically) subharmonic   in $D$, if $U_x$ is (logarithmically) subharmonic in
$D$ for a.e. $x\in X$.
\end{lemma}

Introduce   now the Hardy classes $PL_p\, (0<p< \infty)$ of logarithmically subharmonic
functions in   the unit disc which  play  a  main role in this paper. The class  $PL_p$
contains  all   continuous  logarithmically     subharmonic  functions $U$ in the  unit
disc such that             $M_p(U,r)$  is bounded in $0\le r<1$.  Since $U^p$  is  also
(logarithmically)   subharmonic in the unit disc, in the definition,        instead  of
boundedness we        could                ask for the existence of the  boundary value
$\lim_{r\rightarrow 1^-}M_p(U,r)$.

It  is known that  every  $U\in {PL}_p$     has  the   radial boundary values at almost
every point in $\mathbf{T}$.     The radial  boundary value of $U\in PL_p$  at a  point
$\zeta\in\mathbf{T}$  will be   denoted  (when exists) by $U\sp \ast(\zeta)$, or simply
by $U(\zeta)$.   It may be  proved that $U(\zeta)\in           L^p(\mathbf{T},m_1)$ and
$\log U(\zeta)  \in L^1(\mathbf{T},m_1)$. One  can also  prove the  convergence in mean
$\|U_r - U^*\|_{L^p(\mathbf{T},m_1)}\rightarrow 0$ as $r\rightarrow 1^-$.  We introduce
a norm  (we say ''norm'', but in fact it is  not in a     strong  sense) in ${PL}_p$ by
\begin{equation*}
\|U\|_p\, =\lim_{r\rightarrow 1^-} M_p(U,r)
=  \left\{ \int_{\mathbf{T}} U(\zeta)^p  dm_1(\zeta)\right\}^{1/p}.
\end{equation*}
For introduced   classes            of subharmonic functions  we refer  to the work  of
Privaloff~\cite{PRIVALOV.USSR}.    See also Yamashita's work~\cite{YAMASHITA.PISA}   for
an interesting  consideration        which concerns the   logarithmically   subharmonic
functions of these classes.

We establish now the following useful  extension of the case $n=1$  of our main theorem.

\begin{theorem}\label{TH.LOGSUBHARMONIC}
Assume that $(\dag)$ holds for $\Phi$ and $\mu$.  Let  $U_j(z)\in PL_1\, (j=1,\dots,m)$.
Then
\begin{equation*}
\Phi(U_1,\dots, U_m)\in L^1(\mathbf{U},\mu)
\end{equation*}
with
\begin{equation*}
\int_{\mathbf{U}} \Phi(U_1(z),\dots, U_m(z))\, d\mu(z)\le
 \Phi(\|U_1\|_1,\dots,\|U_m\|_1).
\end{equation*}
Equality   attains if and only if either there exists $j_0,\, 1\le  j_0\le m$ such  that
$U_{j_0}\equiv0$ or each  $U_j\, (j=1,\dots,m)$                 is of the form $$U_j(z)=
|\Psi_{j}^1(z)| ^{\tilde{p}_j}$$                   for some   $(\dots,\Psi_j^1,\dots)\in
\mathcal{E}({\Phi,\mu})$.
\end{theorem}

The following  known   lemma will be useful in the proof of the above result. The proof
is given for   the  sake of completeness.

\begin{lemma}\label{LE.HARDY.CLASSES}
For   every $U(z)\in PL_p$ there exists $f(z)\in H^p$    such that $U(z)\le |f(z)|$ for
$z\in\mathbf{U}$ and $U(\zeta)=|f(\zeta)|$ for  a.e. $\zeta  \in\mathbf{T}$.
\end{lemma}

\begin{proof}
Let     $U(z)\in PL_p$ and suppose w.l.g. that $U\not\equiv 0$. Since  $U(\zeta)\in L^p
(\mathbf{T},m_1)$ and $\log U(\zeta)\in L^1(\mathbf{T},m_1)$,  consider   the     outer
function in $H^p$ given by
\begin{equation*}
f(z)=
\exp \left \{ \int_\mathbf{T} \frac{\zeta+z}{\zeta-z}\, \log U(\zeta)\,
 dm_1(\zeta)\right\}.
\end{equation*}
It is clear that $|f(\zeta)| = U(\zeta)$ for  a.e. $\zeta\in\mathbf{T}$.    Since $\log
U(z)$  is  subharmonic in the unit disc, we have
\begin{equation*}
\log U(z)\le \frac 1{2\pi} \int_0^{2\pi} P(r,\theta-t)\, \log U(e^{it})\, dt
=\log |f(z)|\quad (z=re^{i\theta}).
\end{equation*}
It follows $U(z)\le |f(z)|$ for $z\in\mathbf{U}$.                 Here, $P(r,\theta-t)=
\mathrm {Re}\,(({\zeta+z})/({\zeta-z})),\, \zeta=e^{it}$          is the Poisson kernel.
\end{proof}

\begin{proof}[Proof of Theorem~\ref{TH.LOGSUBHARMONIC}]
W.l.g. assume that  $U_j\not\equiv 0$ for all $j=1,\dots,m$.               According to
Lemma~\ref{LE.HARDY.CLASSES},  there exists (an outer function) $f_j\in H^1$  such that
\begin{equation}\label{TH.LOGSUBHARMONIC.1}
U_j(z)\le |f_j(z)|,\quad  z\in\mathbf{U}
\end{equation}
and
\begin{equation}\label{TH.LOGSUBHARMONIC.2}
U_j(\zeta) = |f_j(\zeta)|\  \ \text{for a.e.}\ \ \zeta\in\mathbf{T}
\end{equation}
In     Theorem~\ref{TH.MAIN.1}        take the above  functions $f_j\, (j = 1,\dots,m)$.
Using~\eqref{TH.LOGSUBHARMONIC.1} and~\eqref{TH.LOGSUBHARMONIC.2},  and monotonicity of
$\Phi$ we obtain
\begin{equation}\label{TH.LOGSUBHARMONIC.BOTH}
\begin{split}
\int_{\mathbf{U}} \Phi(\dots,U_j(z),\dots)\,  d\mu(z)& \le
\int_{\mathbf{U}} \Phi(\dots,|f_j(z)|,\dots)\, d\mu(z)
\\& \le  \Phi(\dots,\|f_j\|_1,\dots) = \Phi(\dots,\|U_j\|_1,\dots),
\end{split}
\end{equation}
what proves  the inequality we need.

Equality holds at the second place of~\eqref{TH.LOGSUBHARMONIC.BOTH}   (regarding   the
equality statement of $(\dag)$) if and  only if    each $f_j\ (j=1,\dots, m)$ is of the
form $f_j= \Psi^1_j$ for some $(\dots,\Psi_j^1,\dots)\in\mathcal{E}(\Phi,\mu)$. In this
case $|f_j(z)|$ does not  vanish anywhere                   in the  unit  disc. In view
of~\eqref{TH.LOGSUBHARMONIC.1}, it   follows that  equality occurs at the   first place
of~\eqref{TH.LOGSUBHARMONIC.BOTH}  if and only  if  $U_j(z)=|f_j(z)|$   for   all $z\in
\mathbf{U}$ and $j=1, \dots,m$.  All together, equality attains at   both places if and
only if
\begin{equation*}
U_j(z)  = |f_j(z) | = |\Psi_j^1(z)|,\quad z\in\mathbf{U}
\end{equation*}
for all $j=1,\dots,m$.
\end{proof}

\begin{remark}
Since
\begin{equation*}
\{U=|f|^p:f\in H^p\}\subseteq {PL}_1
\end{equation*}
note      that Theorem~\ref{TH.LOGSUBHARMONIC} may be seen  as a generalization of the
main theorem   in the case  $n=1$.    The equality statement of Theorem~\ref{TH.MAIN.1}
could        be  derived from the equality statement of Theorem~\ref{TH.LOGSUBHARMONIC}
having in  mind the elementary fact: if $\varphi$ and $\psi$    are analytic functions
in the unit disc and if    $|\varphi(z)|=|\psi(z)|$  for all  $z\in\mathbf{U}$,   then
$\varphi= \alpha \psi$ for  some  constant $\alpha$ of the unit modulus.

Note also that we could  prove Theorem~\ref{TH.LOGSUBHARMONIC}  using only the assumed
condition $(\dag)$ for outer function given     by $U_j^{1/\tilde{p}_j}\in PL_{p_j},\,
j=1,\dots,m$     (in view of Lemma~\ref{LE.HARDY.CLASSES}). That proof would be of the
same length, and it is of some  interest   to observe that we  can       prove the one
dimensional case of our main theorem  without      the Riesz factorization theorem for
Hardy spaces in the unit disc.
\end{remark}

\subsection{A theorem on restricted analytic functions}
Besides Theorem~\ref{TH.MAIN.1} and Theorem~\ref{TH.LOGSUBHARMONIC} for the    complete
proof of Theorem~\ref{TH.MAIN} we will need       some additional results  which are of
interest on their   own right. Here we will firstly  collect  some facts concerning the
Cald{e}r\'{o}n--Zygmund theorem    on iterated limits of analytic functions in the unit
polydisc. We follow mostly the work  of Davis~\cite{DAVIS.TAMS}.

If                         $f\in  N(\mathbf{U}^n)$ ($f\in H_\phi(\mathbf{U}^n))$    and
$(z_{j_1}, \dots,z_{j_k})\in\mathbf{U}^{k},\,1 \le k \le n- 1,\, j_1<\dots<j_k$    then
for the restricted function $f_ {z_{j_1}\dots z_{j_k}}$ there holds   $f_ {z_{j_1}\dots
z_{j_k}}\in N(\mathbf{U}^k)$    ($f_ {z_{j_1}\dots z_{j_k}}\in H_\phi (\mathbf{U}^k ))$.
This  is seen most easily  by         using the  $n$-harmonic  majorant for $\log^+|f|$
($\phi(\log|f|)$).

For                           $f\in N(\mathbf{U}^n)$ and $\zeta_1\in\mathbf{T}$  denote
\begin{equation*}
f_{\zeta_1}(z_2,\dots,z_n)\, =  \lim_{r\rightarrow 1^-}f(r\zeta_1,z_2,\dots,z_n).
\end{equation*}
Whenever this limit  exists,      it defines an analytic function on $\mathbf{U}^{n-1}$.
Zygmund~\cite{ZYGMUND.FUND}   proved:  If $f\in N (\mathbf{U}^n)$ then  $f_{\zeta_1}\in
N(\mathbf{U}^{n-1})$  for almost every $\zeta_1\in\mathbf{T}$.          For $\zeta_1\in
\mathbf{T}$                    satisfying the  previous condition  we may then consider
\begin{equation*}
f_{\zeta_1\zeta_2}(z_3,\dots,z_n) \,
= \lim_{r\rightarrow 1^-} f_{\zeta_1} (r\zeta_2,z_3,\dots,z_n)\in N(\mathbf{U}^{n-2})
\end{equation*}
for    almost every $\zeta_2\in \mathbf{T}$.     Continuing in this manner we arrive at
\begin{equation*}
f_{\zeta_1\dots \zeta_n}
= \lim_{r\rightarrow 1^-} f_{\zeta_1\dots\zeta_{n-1}} (r\zeta_n),
\end{equation*}
whenever  this limit exists. In~\cite{DAVIS.TAMS} Davis showed  that if  $f\in N\sp\ast
(\mathbf{U}^n)$, then    the iterated  limits of $f$ are almost everywhere  independent
of the order of iteration. In fact, the iterated limit and the radial   limit are equal
almost everywhere.         A similar theorem holds for functions belonging to the space
$H_\phi$ (as a   consequence).  Zygmund~\cite{ZYGMUND.FUND} proved the preceding result,
but for $f\in N_{n-1}(\mathbf{U}^n)\supseteq N\sp\ast(\mathbf{U}^n)$.           For the
description of the class $N_{n-1}(\mathbf{U}^n)$           see the Zygmund paper.  Then
Calder\'{o}n and Zygmund posed the question whether $N_{n-1}(\mathbf{U}^n)$ may      be
replaced by $N(\mathbf{U}^n)$. The result of Davis gives a partial answer enough    for
our needs.

In the sequel  we will need only  the following proposition. A proof follows from   the
above discussion.

\begin{proposition}\label{TH.CALDERON.ZYGMUND}
Let $f(z)=f(z_1,\dots,z_n)\in H_\phi(\mathbf{U}^n)$.     Then for an integer $1\le k\le
n-1$ and mutually disjoint integers $\{j_1,\dots, j_k\}\subseteq\{1,\dots,n\}$  we have
\begin{enumerate}
\item    $f_{z_{j_1}\dots z_{j_k}}\in H_\phi(\mathbf{U}^{n-k})$ for all $(z_{j_1},\dots,
z_{j_k})\in\mathbf{U}^{k}$;
\item  $f_{\zeta_{j_1}\dots \zeta_{j_k}}\in H_\phi(\mathbf{U}^{n-k})$ for almost  every
$(\zeta_{j_1},\dots,\zeta_{j_k})\in\mathbf{T}^{k}$.
\end{enumerate}
Particularly,  the iterated boundary function $f_{\zeta_{j_1}\dots \zeta_{j_k}}$ is the
same as the  radial boundary function $f\sp\ast_{\zeta_{j_1}\dots \zeta_{j_k}}$     for
almost every $(\zeta_{j_1},\dots,\zeta_{j_k})\in \mathbf{T}^k$.
\end{proposition}

\begin{remark}
It  follows  that the function  $f_{\zeta_{j_1}\dots \zeta_{j_k}}$ which appears in the
above proposition does not depend  on the order of iteration; thus, we may assume $1\le
j_1<\dots<j_k \le n$.
\end{remark}

We introduce                    $PL_p(\mathbf{U}^n)\, (0<p<\infty)$ as the class of all
$n$-logarithmically subharmonic functions in $\mathbf{U}^n$ such that
\begin{equation*}
\|U\|_p\  = \, \sup_{0\le r<1} M_p(U,r)\,
= \sup_{0\le r<1}\left\{ \int _{\mathbf{U}^n}\, U(r\zeta)^p\, dm_n(\zeta)\right\}^{1/p}
\end{equation*}
It is       known that every $U\in PL_p(\mathbf{U}^n)$ has the radial limit a.e., i.e.,
there exists
\begin{equation*}
U(\zeta) = \lim_{r\rightarrow 1} U(r\zeta)\quad\text{for a.e.}\ \zeta\in\mathbf{T}^n.
\end{equation*}
Obviously   $\left\{|f|:  f\in H^p(\mathbf{U}^n)\right\} \subseteq PL_p(\mathbf{U}^n)$.

This subsection  is devoted to the proof of the following

\begin{theorem}\label{TH.INT.LOGSUB.GENERAL}
For   $f(z)=f(z_1,\dots,z_n)\in H^p(\mathbf U^n)$ and integers  $1\le j_1<\dots<j_k\le
n$, where $1\le k<n$ is also an integer, the function
\begin{equation*}
U(z_{j_1},\dots,z_{j_k})=\|f_{z_{j_1}\dots z_{j_k} }\|_p^p
\end{equation*}
is well defined in $\mathbf{U}^{k}$.        Moreover, $U\in  PL_1(\mathbf{U}^k)$   and
the  $PL_1$-norm of $U$ is given by
\begin{equation*}
\|U\|_1=\|f\|_p^p.
\end{equation*}
\end{theorem}

For a proof of Theorem~\ref{TH.INT.LOGSUB.GENERAL} we need  the auxiliary results that
are  formulated  below.

\begin{lemma}\label{LE.VUKOTIC}        For $z\in\mathbf{U}^n$ there holds the (optimal)
estimate
of the modulus
\begin{equation*}
|F(z)|^p\le \frac 1 {(1-|z_1|^2) \cdots(1-|z_n|^2) }\|F\|_p,
\end{equation*}
of     $F(w)=F(w_1,\dots,w_n)\in H^p(\mathbf{U}^n)\, (0<p<\infty)$. The equality  sign
occurs if and only if
\begin{equation*}
F(w) = \lambda K_z(w)^{2/p},
\end{equation*}
where                                                 $\lambda$ is any constant. Here,
$K(w,z)=\prod_{j=1}^n ({1-w_j \overline{z}_j})^{-1}$  stands for the Cauchy--Szeg\"{o}
kernel for the unit polydisc.
\end{lemma}

\begin{remark}
A similar     statement to Lemma~\ref{LE.VUKOTIC}      may be found in the Vukoti{\'c}
work~\cite{VUKOTIC.PAMS} for functions which belong to weighted Bergman spaces in the
unit ball of  $\mathbf{C}^n$. The same idea may  be applied  to derive the  preceding
optimal  result.
\end{remark}

In what follows we will use several times the next lemma  which follows   immediately
from the Lebesgue Dominant Convergence Theorem. If $(X,\nu)$ is a measurable    space
and $\Lambda$  a non--empty set of indexes,  we say that $G$  is    dominant  for the
family $\{F_\alpha:\alpha\in\Lambda\}$ of  real--valued measurable functions in   $X$
if $F_\alpha(z)\le G(z)$ for almost     every $z\in X$ and for all $\alpha\in\Lambda$.

\begin{lemma}\label{LE.CONTINUOUS}
Let $D$ be a domain,   $(X,\nu)$  a measurable space, and let  $F(z,w)$ be defined in
$D\times X$ such that   $F_z   \in L^1(X,\nu)$ for all $z\in D$ and $F_w\in C(D)$ for
almost every  $w\in   X$.   If  for any compact set $K\subseteq D$ there exists $G\in
L^1(X,\nu)$  dominant for the family
\begin{equation*}
\{F_z:z\in K\},
\end{equation*}
then
\begin{equation*}
z\rightarrow \int_X F(z,w)\, d\nu(w)
\end{equation*}
is continuous in $D$.
\end{lemma}

\begin{proof}[Proof of Theorem~\ref{TH.INT.LOGSUB.GENERAL}]
For                      simplicity we will assume that  ${j_1}= n-k+1$ and $j_k = n$.

In  view of the first part of Proposition~\ref{TH.CALDERON.ZYGMUND},     the value of
\begin{equation*}\begin{split}
U(z_{n-k+1},\dots,z_{n})&
=\|f_{z_{n-k+1}\dots z_{n} }\|_p^p
\\& = \int_{\mathbf{T}^{n-k}} |f^*_{z_{n-k+1}\dots z_{n}}(\zeta_1,\dots,\zeta_{n-k}) |^p dm_{n-k}(\zeta_1,\dots,\zeta_{n-k})
\\& = \int_{\mathbf{T}^{n-k}} |f^*_{\zeta_1\dots\zeta_{n-k}}(z_{n-k+1},\dots,z_{n}) |^p dm_{n-k}(\zeta_1,\dots,\zeta_{n-k})
\end{split}
\end{equation*}
is finite for every  $(z_{n-k+1},\dots,z_{n})\in  \mathbf{U}^k$.     We have used the
identity
\begin{equation*}
f\sp\ast_{z_{n-k+1}\dots z_{n}}(\zeta_1,\dots,\zeta_{n-k})
 =f\sp\ast_{\zeta_1\dots\zeta_{n-k}}(z_{n-k+1},\dots,z_{n}),
\end{equation*}
which holds  for  every  $(z_{n-k+1},\dots,z_{n})\in\mathbf{U}^k$     and for  almost
every $(\zeta_1,\dots,\zeta_{n-k})\in\mathbf{T}^{n-k}$.   Let us say that on the left
side of the preceding           relation         we mean the radial boundary value of
$f_{z_{n-k+1}\dots  z_{n}}$ at $(\zeta_1,\dots,\zeta_{n-k})\in\mathbf{T}^{n-k}$.   On
the               right side we have the value of the        radial boundary function
$f\sp\ast_{\zeta_1\dots\zeta_{n-k}}$   at the point       $(z_{n-k+1},\dots,z_{n})\in
\mathbf{U}^k$.                           In both cases it is about the boundary value
\begin{equation*}
\lim_{r\rightarrow 1^-}f(r\zeta_1,\dots, r\zeta_{n-k},z_{n-k+1},\dots,z_{n}).
\end{equation*}

Denote $z'=(z_1,\dots,z_{n-k})$ and  $z''=(z_{n-k+1},\dots,z_{n})$. Similarly, denote
$\zeta' =(\zeta_1,\dots,\zeta_{n-k})$ and $\zeta'' =(\zeta_{n-k+1},\dots, \zeta_{n})$.

For all $0\le r <1$ introduce
\begin{equation*}
\tilde{U}_r(z'')\, =
\int_{\mathbf{T}^k}  |f_{ r \zeta'}(z'')|^p dm_k(\zeta'),   \quad z''\in \mathbf{U}^k.
\end{equation*}
Note                that  $\tilde{U}_r (z'')$ is not  the  $r$-dilatation of $U(z'')$.

Regarding         Lemma~\ref{LE.RONKIN.2}, the function  $\tilde{U}_r,\, 0\le r<1$ is
$k$-logarithmically subharmonic in the polydisc  $\mathbf{U}^k$, since the    same is
true for
\begin{equation*}
\mathbf{U}^k\ni z'' \rightarrow |f(r \zeta',z'')|^p,
\end{equation*}
for all  $\zeta' = (\zeta_1,\dots,\zeta_{n-k})\in \mathbf{T}^{n-k}$.  The convergence
\begin{equation*}
\tilde{U}_r(z'')\rightarrow U(z''),\quad r\rightarrow 1^-
\end{equation*}
is nondecreasing in $r$,     if $z''=(z_{n-k+1},\dots,z_{n})\in\mathbf{U}^k$ is fixed,
since
\begin{equation*}
\tilde{U}_r(z'') = M_p^p(f_{z''},r)
\end{equation*}
and  $f_{z''}\in H^p(\mathbf{U}^{n-k})$. Therefore,
\begin{equation*}
U(z'')\, = \sup_{0\le r<1} \tilde {U}_r(z'')
\end{equation*}
for all  $z''=(z_{n-k+1}, \dots,z_{n})\in\mathbf{U}^k$.

In view of Lemma~\ref{LE.RONKIN.1},         in order  to show that  $U$ is continuous
$k$-logarithmically subharmonic in      $\mathbf{U}^k$, i.e., that it  belongs to the
class  $PL(\mathbf{U}^k)$, it is enough to prove that  $U$  is continuous in the same
domain. To realize that we will use Lemma~\ref{LE.CONTINUOUS}. Let $K$ be any compact
subset of $\mathbf{U}^k$.  We  show that there exists an integrable function which is
dominant  for the  family
\begin{equation*}
\left\{|f\sp\ast_{z''}(\zeta')|^p :z'' = (z_{n-k+1},\dots,z_n)\in K,\,
 \zeta'=(\zeta_1,\dots,\zeta_{n-k})\in\mathbf{T}^{n-k} \right\}.
\end{equation*}
There exists a constant   $C=C (K)$ such that
\begin{equation}\label{INEQ.CK.ESTIMATE}
 \frac {1}{(1-|z_{n-k+1}|^2)\cdots (1-|z_n|^2)} \le C
\end{equation}
for $z''=(z_{n-k+1},\dots,z_n)\in K$.   Since $f\sp\ast_{\zeta_1\dots\zeta_{n-k}}\in
H^p(\mathbf{U}^{k})$      for a.e. $(\zeta_1,\dots, \zeta_{n-k})\in\mathbf{T}^{n-k}$
(see the second part of Proposition~\ref{TH.CALDERON.ZYGMUND}),     according to the
Lemma~\ref{LE.VUKOTIC},  we find
\begin{equation}\label{INEQ.ESTIMATE}
|f\sp\ast_{\zeta_1\dots \zeta_{n-k}}(z_{n-k+1},\dots,z_n)|^p\le\frac
{\|f\sp\ast_{\zeta_1\dots \zeta_{n-k}}\|^p_p }{(1-|z_{n-k+1}|^2)\cdots (1-|z_n|^2)},
\end{equation}
for $(z_{n-k+1},\dots,z_n)\in \mathbf{U}^k$.                              The growth
estimate~\eqref{INEQ.ESTIMATE}                gives the dominant function. Denote by
\begin{equation*}\begin{split}
{G} (\zeta_1,\dots, \zeta_{n-k}) &= C\|f\sp\ast_{\zeta_1\dots \zeta_{n-k}}\|^p_p
\\&=C\int_{\mathbf{T}^k}|f\sp\ast_{\zeta_1\dots \zeta_{n-k}}(\zeta_{n-k+1},\dots,\zeta_{n})|^p dm_{k}(\zeta_{n-k+1},\dots,\zeta_n),
\end{split}\end{equation*}
a function defined almost everywhere on the torus  $\mathbf{T}^{n-k}$.  The function
${G}$ is integrable, since applying the                          Fubini  theorem and
Proposition~\ref{TH.CALDERON.ZYGMUND}, we find
\[\begin{split}
\int_{\mathbf{T}^{n-k}} {G} (\zeta_1,\dots, \zeta_{n-k})=
C\int_{\mathbf{T}^n}|f\sp\ast(\zeta_1,\dots,\zeta_n)|^p  dm_n(\zeta_1,\dots,\zeta_n)
= C\|f\|^p_p<\infty.
\end{split}\]
Using~\eqref{INEQ.ESTIMATE}                 and~\eqref{INEQ.CK.ESTIMATE}, it follows
\begin{equation*}
|f\sp\ast_{z_{n-k+1}\dots z_n}(\zeta_1,\dots,\zeta_{n-k})|^p
\le G({\zeta_1,\dots,\zeta_{n-k}})
\end{equation*}
for         almost every  $({\zeta_1,\dots,\zeta_{n-k}})\in\mathbf{T}^{n-k}$ and for
every $(z_{n-k+1},\dots,z_n)\in K$. Since the  function $|f\sp\ast_{\zeta'}(z'')|^p$
is continuous (for  a.e. $\zeta' =({\zeta_1,\dots,\zeta_{n-k}})\in \mathbf{T}^{n-k}$,
as modulus of an analytic function,                                what follows from
Proposition~\ref{TH.CALDERON.ZYGMUND}), applying Lemma~\ref{LE.CONTINUOUS} we obtain
that
\begin{equation*}\begin{split}
U(z_{n-k+1},\dots,z_{n}) \,
= \int_{\mathbf{T}^{n-k}} |f\sp\ast_{z_{n-k+1}\dots z_{n}}(\zeta_1,\dots,\zeta_{n-k}) |^p d m_{n-k}(\zeta_1,\dots,\zeta_{n-k})
\end{split}
\end{equation*}
is also continuous in $\mathbf{U}^{k}$.

In order to finish the proof of this theorem              we have to show that $U\in
{PL}_1(\mathbf{U}^k)$ and $\|U\|_1=\|f\|_p^p$.  First of all, for $0\le r<1$ we have
\[\begin{split}
M_1(U,r) &\, = \int_{\mathbf{T}^k} \|f_{r\zeta''}\|_p^p\, dm_k(\zeta'')
\\&= \int_{\mathbf{T}^{k}} \left\{\int_{\mathbf{T}^{n-k}} |f\sp\ast_{r\zeta''}(\zeta')|^pd m_{n-k}(\zeta')\right\}dm_k(\zeta'')
\\&= \int_{\mathbf{T}^{n-k}} \left\{\int_{\mathbf{T}^{k}} |f\sp\ast_{r\zeta''}(\zeta')|^p dm_k(\zeta'')\right\}d m_{n-k}(\zeta')
\\&= \int_{\mathbf{T}^{n-k}} M_p^p (f\sp\ast_{\zeta'},r)\,  d m_{n-k}(\zeta')
\\& \le \int_{\mathbf{T}^{n-k}}\|f\sp\ast_{\zeta'}\|_p^p \, d m_{n-k}(\zeta')
\\&= \|f\|^p_p<\infty.
\end{split}\]
It follows
\begin{equation*}
\|U\|_1 \, = \sup_{0\le  r<1} M_1(U,r)\le\|f\|_p^p.
\end{equation*}
Therefore,   we indeed have     $U\in {PL}_1(\mathbf{U}^k)$.

To   show the reverse inequality, it is enough to apply the Fatou lemma. Since $U\in
PL_1$, the radial boundary value $U(\zeta'')$ exists for almost every point $\zeta''
=(\zeta_{n-k+1},\dots,\zeta_{n})\in\mathbf{T}^{k}$,                     and it holds
\begin{equation*}\begin{split}
U(\zeta_{n-k+1},\dots,\zeta_{n})\,  & =
\lim_{r\rightarrow 1^- }U(r\zeta_{n-k+1},\dots,r\zeta_{n})
\\&=\lim_{r\rightarrow 1^-}\|f_{(r\zeta_{n-k+1})\dots (r\zeta_{n}) }\|_p^p
\ge\|f\sp\ast_{\zeta_{n-k+1}\dots \zeta_{n} }\|_p^p.
\end{split}\end{equation*}
Now,  we obtain
\begin{equation*}\begin{split}
\|U\|_1& \, = \int_{\mathbf{T}^k} U(\zeta'')\, dm_k(\zeta'')
 \ge \int_{\mathbf{T}^k}\|f\sp\ast_{\zeta''}\|_p^p\, dm_k(\zeta'')=\|f\|_p^p.
\end{split}\end{equation*}
what finishes this proof.
\end{proof}

\section{The proof of the main  theorem}
This section  is devoted to                            the proof of the main theorem.

\begin{proof}[Proof of Theorem~\ref{TH.MAIN}]
We use the induction on the dimension parameter $n$.       For $n=1$ we have already
proved the theorem (Theorem~\ref{TH.MAIN.1}).

Assume  now that this theorem is valid for $n-1$, where $n\ge 2$. We are    going to
prove it for $n$. Let $f_j(z)=f_j(z',z_n)\in H^{p_j}(\mathbf{U}^n)$          for all
$j=1,\dots,m$, where    we have denoted $z'=(z_1,\dots,z_{n-1})$.  Since $d\nu_n(z)=
d\nu_{n-1}(z')\times d\mu(z_n)$,               applying the Fubini theorem we obtain
\begin{equation}\label{BOTH.INEQ}\begin{split}
\int_{\mathbf {U}^n} \Phi(\dots,|f_j(z)|^{p_j},\dots) & \,d\nu_n(z)
\\&=\int_{\mathbf{U}} \left\{\int_{\mathbf{U}^{n-1}} \Phi(\dots,|f_j^{z_n}(z')|^{p_j},\dots)\, d{\nu}_{n-1}(z')\right\} d\mu(z_n)
\\&\le\int_{\mathbf{U}} \Phi(\dots,\|f_j^{z_n}\|_{p_j}^{p_j},\dots) \, d\mu(z_n)
\le \Phi(\dots,\|f_j\|_{p_j}^{p_j},\dots),
\end{split}\end{equation}
what proves the inequality in our theorem.

More precisely, since $f^{z_n}_j\in H^{p_j}(\mathbf{U}^{n-1})$  for every fixed $z_n
\in\mathbf{U}$, applying   the induction hypothesis, we obtain the  first inequality
above
\begin{equation}\label{FIRST.INEQ}
\int_{\mathbf {U}^{n-1}}  \Phi(\dots,|f_j^{z_n}(z')|^{p_j},\dots)\, d{\nu}_{n-1}(z')
\le\Phi(\dots,\|f_j^{z_n}\|_{p_j}^{p_j},\dots).
\end{equation}
In view of Theorem~\ref{TH.INT.LOGSUB.GENERAL}, let $U_j\in {PL}_1\,  (j=1,\dots,m)$
be defined in the unit disc  in the following way
\begin{equation}\label{GJ.NORM}
U_j(z_n) = \|f_j^{z_n}\|^{p_j}_{p_j},\quad z_n\in\mathbf{U}.
\end{equation}
By the same theorem we have
\begin{equation}
\|U_j\|_1=\|f_j\|_{p_j}^{p_j}
\end{equation}
for all $j=1,\dots,m$. The second inequality
\begin{equation}\begin{split}\label{SECOND.INEQ}
\int_{\mathbf{U}} \Phi(\dots,\| f_j^{z_n} \|_{p_j}^{p_j},\dots) \, d\mu(z_n)
&=\int_{\mathbf{U}} \Phi(\dots,U_j(z_n),\dots) \, d\mu(z_n)
\\&\le \Phi(\dots,\|U_j\|_1,\dots)  =\Phi(\dots,\|f_j\|_{p_j}^{p_j},\dots)
\end{split}\end{equation}
follows  from    the  logarithmically subharmonic version of the case $n=1$ of   our
main  theorem   (Theorem~\ref{TH.LOGSUBHARMONIC}).

Regarding Theorem~\ref{TH.LOGSUBHARMONIC},  it is not hard to see that   the    main
inequality holds if we take  $U_j\in PL_1(\mathbf{U}^n)$    instead of $f_j\in H^{1}
(\mathbf{U}^n)$ (for all $j=1,\dots,m$) with $PL_1(\mathbf{U}^n)$-norms of functions
$U_j\, (j=1,\dots,m)$ on the right side.   This is important to note for the rest of
this proof which  is devoted to the extremal  functions. To establish   this version
of  the main inequality, it is enough to prove it for dilatations $(U_j)_r,\, 0\le r
<1$. One can prove this as we have just done  for analytic   functions  (even easily,
since we  have continuous  functions  in  a vicinity  of the unit polydisc). Letting
$r\rightarrow 1^-$ we obtain the desired inequality,                       since the
$PL_1(\mathbf{U}^n)$-norms of dilations                   of a function in the class
$PL_1(\mathbf{U}^n)$ converge to the $PL_1(\mathbf{U}^n)$-norm of the  same function.
This approach,  however, does not provide  all extremal       functions for the main
inequality in the  subharmonic case neither in the analytic case.

As we have said,    the     second half of this proof is    devoted  to the extremal
functions  for the main inequality. However, we will make a digression   in order to
show that the  function
\begin{equation*}
F(z_n)
= \int_{\mathbf{U}^{n-1}}   \Phi(\dots,|f_j^{z_n}(z')|^{p_j},\dots)\, d\nu_{n-1}(z'),
\end{equation*}
which   appears at the beginning of this proof,    under the integral  sign   of the
relation~\eqref{FIRST.INEQ} is continuous  at every point  $z_n\in\mathbf{U}$.    In
view of Lemma~\ref{LE.CONTINUOUS}, it is enough to show  that for any compact set $K
\subseteq\mathbf{U}$                                          there exists $G(z')\in
L^1(\mathbf{U}^{n-1},\nu_{n-1})$ dominant                             for the family
\begin{equation*}
\left\{\tilde{F}_{z_n}(z')
=  \Phi(\dots,|f_j^{z_n}(z')|^{p_j},\dots): z_n\in K\right\}.
\end{equation*}
To        finish  this, let  a  constant $C = C(K)$  be chosen in such  a   way that
\begin{equation*}
\frac1{1-|z_n|^2}\le C,\quad z_n\in K.
\end{equation*}
Since             $f_j^{z'}\in H^{p_j}$  for every fixed  $z'\in\mathbf{U}^{n-1}$ by
Lemma~\ref{LE.VUKOTIC}, we obtain the estimate
\begin{equation*}
|f_j^{z'}(z_n)|^{p_j}
\le \frac {1}{ 1-|z_n|^2 }\|f_j^{z'}\|_{p_j}^{p_j},\quad z_n\in\mathbf{U}.
\end{equation*}
Thus, for fixed $z_n\in K$ we have
\begin{equation*}
\tilde{F}_{z_n}(z') = \Phi(\dots,|f_j(z',z_n)|^{p_j},\dots)\le
 \Phi(\dots,C\tilde{U}_j(z'),\dots),
\end{equation*}
where we have denoted
\begin{equation*}
\tilde{U}_j(z')=\|f_j^{z'}\|_{p_j}^{p_j},\quad z'\in\mathbf{U}^{n-1}
\end{equation*}
for          each $j=1,\dots,m$; see also Theorem \ref{TH.INT.LOGSUB.GENERAL}.  Thus,
\begin{equation*}
\tilde{F}_{z_n}(z')\le G(z'),
\end{equation*}
for
\begin{equation*}
G(z')= \Phi(\dots, C\tilde{U}_j(z'),\dots),\quad z'\in\mathbf{U}^{n-1}.
\end{equation*}
It remains to   show that $G\in L^1(\mathbf{U}^{n-1},\nu_{n-1})$. Since $\tilde{U}_j
\in PL_1(\mathbf{U}^{n-1})$ for all $j=1,\dots,m$, according  to the main inequality
for  functions in the   class $PL_1(\mathbf{U}^{n-1})$,                    we obtain
\[\begin{split}
\int_{\mathbf{U}^{n-1}} G(z')\, d\nu_{n-1}(z') &
 =  \int_{\mathbf{U}^{n-1}} \Phi(\dots,C\tilde{U}_j(z'),\dots)\, d\nu_{n-1}(z')
\\&\le \Phi(\dots,C\|\tilde{U}_j\|_1,\dots) < \infty.
\end{split}\]

Let          us now prove the second half of our theorem  --  the equality statement.

Equality      attains in the main inequality if and only if  equality holds at  both
places in~\eqref{BOTH.INEQ}. If $f_j\equiv 0$ for some $j,\, 1\le j\le m$,      then
equality obvious holds at both places in~\eqref{BOTH.INEQ}. In  the sequel we assume
this is not the case and will firstly prove that if equality attains   in   the main
inequality, then each $f_j\,  (j= 1,\dots,m)$    does not  vanish in  $\mathbf{U}^n$.
After that we  will use the   equality part  of the assumption $(\dag)$ in order  to
derive the equality statement in general.

Since $f_j\not\equiv 0$, in  view of~\eqref{GJ.NORM}  we     have $U_j\not \equiv 0$
for all $j=1,\dots,m$. Thus, equality holds at the second place  in~\eqref{BOTH.INEQ}
if                                                                       and only if
\begin{equation*}
U_j(z_n) =  | \Psi_j^1(z_n)|, \quad z_n\in\mathbf{U}
\end{equation*}
for all $j=1,\dots,m$,  where $(\dots,\Psi_j^1,\dots)\in\mathcal{E}(\Phi,\mu)$  (see
the equality statement of Theorem~\ref{TH.LOGSUBHARMONIC}). This     means that each
$U_j$ does not vanish in the unit  disc.

Note now that the        first   inequality in~\eqref{BOTH.INEQ} may be rewritten in
the    following    form
\begin{equation*}
\int_{\mathbf{U}}  F(z_n)\, d\nu _{n-1}(z_n)\le
\int_{\mathbf{U}} \Phi(\dots,U_j(z_n),\dots)\, d\nu_{n-1}(z_n).
\end{equation*}
Since        $F(z_n)$ and  $\Phi(\dots,U_j(z_n),\dots)$ are continuous in all points
$z_n\in\mathbf{U}$,       equality occurs in the preceding inequality if and only if
\begin{equation*}
F(z_n)         =    \Phi(\dots,U_j(z_n),\dots)\ \ \text{for all}\ \ z_n\in\mathbf{U}.
\end{equation*}
This  means         that equality holds  in~\eqref{FIRST.INEQ}  also for all $z_n\in
\mathbf{U}$,            what is  possible (in view of the  equality statement of the
induction hypothesis) if and only if
\begin{equation*}
f_j^{z_n}(z') = \Psi^{n-1}_{j,z_n}(z'),\quad z'\in\mathbf{U}^{n-1}
\end{equation*}
for all $j=1,\dots,m$;                     here $(\dots,\Psi^{n-1}_{j,z_n},\dots)\in
\mathcal{E}(\Phi,\nu_{n-1})$ for $z_n\in\mathbf{U}$.              Namely, since each
$U_j(z_n)\, (j= 1,\dots,m)$  does  not vanish in the unit   disc, it is not possible
to exist $j,\, 1\le j\le m$ and $z_n\in\mathbf{U}$  such that   $f_{j}^{z_n}\equiv0$
(see~\eqref{GJ.NORM}).

All together,   we have proved that if equality holds in the main inequality,   then
each $f_j\, (j=1,\dots,m)$ does not vanish in $\mathbf{U}^n$. Thus, we can      take
some branch   $f_j^{{p_j}/{\tilde{p}_j}}\in H^{\tilde{p}_j}(\mathbf{U}^n)$   for all
$j=1,\dots,m$.                   Using    now the equality statement of $(\dag)$ for
$f_j^{{p_j}/\tilde{p}_j},\,   j=1,\dots,m$, we find that  equality    holds if   and
only  if   each $f_j\, (j=1,\dots,m)$     is of the form   $f_j^{{p_j}/\tilde{p}_j}=
\Psi_j^n$,                                                       or what is the same
\begin{equation*}
f_j=\{\Psi_j^n\}^{\tilde{p}_j/p_j}
\end{equation*}
for some       $(\dots,\Psi_j^n,\dots)\in\mathcal{E}(\Phi,\nu_n)$. This finishes the
proof  of the  equality statement of Theorem~\ref{TH.MAIN}.
\end{proof}

\begin{remark}
An inner function in $\mathbf{U}^n$ is a bounded analytic      function whose radial
boundary  values  satisfy $|G\sp\ast(\zeta)| = 1$  a.e. on     $\mathbf{T}^n$.    An
inner function $G$ in $\mathbf{U}^n$   is said to be good if   $\mathcal {H}_G\equiv
0$; here, $\mathcal{H}_G$ stands for the least $n$-harmonic   majorant of $\log |G|$
in $\mathbf{U}^n$. One can prove that  an inner function $G$ is good if  and only if
$\mathcal{H}_G (0) = 0$. In the case of the  unit disc,  the good    inner functions
are  precisely the Blaschke products.

As is well known,     in  one  variable case, there corresponds to every  $f\in H^p$
a Blaschke product $B$ such that  $h = f/B$ has no zeros in  the   unit disc,  $h\in
H^p(\mathbf{U})$, and even $\|h\|_p  = \|f\|_p$. If $\mathbf{U}$     is  replaced by
$\mathbf{U}^n$, where $n>1$, one might expect that the role of the Blaschke products
is taken over by    the good inner functions. This  is true, but   only to a certain
extent. The analogues of the one--variable      theory hold for exactly  those $f\in
H^p(\mathbf{U}^n)$ for which the  least $n$-harmonic majorant     $\mathcal{H}_f$ of
$\log |f|$ is          the real part of an analytic function -- the class denoted by
$\mathrm{RP}(\mathbf{U}^n)$. We have

\begin{proposition}[cf.~\cite{RUDIN.BOOK.POLYDISC}]
Assume $f \in N(\mathbf{U}^n)$. Then
\begin{enumerate}
\item If $\mathcal{H}_f$ is not in $\mathrm{RP}(\mathbf{U}^n)$,   then no good inner
function has the same zeros as $f$;
\item If $\mathcal{H}_f\in \mathrm{RP}(\mathbf{U}^n)$, then there   is a  good inner
function  (unique up to a multiplicative constant) with the same zeros        as $f$.
\end{enumerate}
\end{proposition}

And

\begin{proposition}[cf.~\cite{RUDIN.BOOK.POLYDISC}]
Let $G$ be a good inner  function, $h$ an  analytic function in $\mathbf{U}^n$,  and
$f =  G \cdot h$. Then  $h\in N(\mathbf{U}^n)$. Moreover,
\begin{enumerate}
\item           If $f\in N\sp\ast(\mathbf{U}^n)$, then $h\in N\sp\ast(\mathbf{U}^n)$.
\item If $f\in H^p(\mathbf{U}^n)\, (0 <p<  \infty)$,   then $h\in H^p(\mathbf{U}^n)$,
and $\|h\|_p = \|f\|_p$.
\end{enumerate}
\end{proposition}

We will    sketch now a simple but incomplete proof of Theorem~\ref{TH.MAIN} based on
the preceding    propositions.  This proof is motivated by the method which is almost
standard in the    theory of Hardy spaces (in the unit disc).    However, as we  have
seen,  this approach    gives a complete proof in the    one--dimensional case of our
main theorem (Theorem~\ref{TH.MAIN.1}).

Let $f_j\in H^{p_j}(\mathbf{U}^n),\, j=1,\dots,m$.    Without     loss of  generality,
suppose that  $f_j\not\equiv0$ for all $j=1,\dots,m$. Assume, moreover,     that each
$\log|f_j|$          has an $n$-harmonic               majorant which      belongs to
$\mathrm{RP}(\mathbf{U}^n)$. With  the  preceding additional    assumption,     it is
possible to obtain the    factorization $f_j = G_j      h_j$,  where  $G_j$ is a good
inner function on the  unit polydisc with the same zeroes  as $f_j$.  Recall, we take
$G_j \equiv 1$, if $f_j$ is zero--free. Now the proof goes in the same way  as in the
case  $n=1$.
\end{remark}

\section{On the Burbea  inequality}
\subsection{An extension of the Burbea inequality}
Our       goal here is to derive the next theorem proved by  the author in the   case
of the unit bidisc. The inequality in the case $m=2$ may be  found       in the Kalaj
paper~\cite{KALAJ.ANNALI}.                                           It is also given
in~\cite{MATELJEVIC.PAVLOVIC.PUBLICATIONS.1991}  in the non--weighted case, i.e., for
$m=2$ and $f_1=f_2$.

\begin{theorem}[cf.~\cite{MARKOVIC.PAMS}]\label{TH.BURBEA.MARKOVIC}
Let $m\ge 2$ be an integer and $f_j(z)\in H^{p_j}(\mathbf U^n)\, (0<p_j<\infty)$  for
all $j=1,\, 2, \dots,m$. Then
\begin{equation*}
\prod_{j=1}^m |f_j|^{p_j}\in L^1(\mathbf{U}^n,dA_{\mathbf{m}-\mathbf{2}})
\end{equation*}
with
\begin{equation*}
\int_{\mathbf U^n} \left\{\prod_{j=1}^m |f_j(z)|^{p_j}\right\}       dA_{\mathbf{ m}-
\mathbf{2}}(z)\le \prod_{j=1}^m  \|f_j\|^{p_j}_{p_j}.
\end{equation*}
Equality           attains if and only if either $\prod_{j=1}^mf_j\equiv 0$   or each
$f_j\,  (j=1,\, 2,\dots,m)$ is of the form
\begin{equation*}
f_j(z) = c_j{K}_w(z)^{2/{p_j}},\quad z\in\mathbf{U}^{n}
\end{equation*}
for         some (common) $w\in\mathbf{U}^n$ and a non--zero $c_j$.  Here, ${K}(z,w)=
\prod_{j=1}^m(1-z_j\overline{w}_j)^{-1}$ denotes the Cauchy--Szeg\"{o} kernel for the
unit polydisc.
\end{theorem}

\begin{proof}
The general case which contains this corollary follows from our Theorem~\ref{TH.MAIN},
where we have to take  for $\Phi:\mathbf{R}^m_+\rightarrow\mathbf{R}_+$ and $\mu$ the
following:
\begin{equation*}
\Phi(x_1,\dots,x_m)=  \prod_{j=1}^m x_j
\end{equation*}
and
\begin{equation*}
d\mu (z) = dA_{m-2} (z) =   \frac{m-1}\pi (1-|z|^2)^{m-2}dA(z).
\end{equation*}
The condition $(\dag)$ is satisfied for $\tilde{p}_j =2$ (for all $j=1,\, 2,\dots,m)$,
and the family of extremals for $\nu_n$ is
\begin{equation*}
\mathcal{E}(\Phi,\nu_n)
=\{(C_1K_w(z),\dots, C_m K_w(z) ) : w\in\mathbf{U}^n,\, C_j\ne 0,\, j=1,\ 2,\dots,m\},
\end{equation*}
according to Theorem~\ref{TH.GENERAL.HARDY.POLYDISC}    (for $q_j=\mathbf{1},\, j=1,\,
2,\dots,m$).

Recall that $H_{\mathbf {1}} (\mathbf{U}^n)=H^2(\mathbf{U}^n)$ is the Hardy space and
$H_{\mathbf{m}}(\mathbf{U}^n)=L^2_{a,m-2}(\mathbf{U}^n)$      is the weighted Bergman
space.
\end{proof}

\subsection{Remarks}
The rest of this section is devoted to some comments concerning  the particular cases
of  the preceding theorem.

For    $n=1$     and $p_j=p,\, f_j=f\in H^p,\, j= 1,\, 2,\dots,m$  the  inequality in
Theorem~\ref{TH.BURBEA.MARKOVIC}       becomes
\begin{equation}\label{INEQ.BURBEA.MATELJEVIC.PAVLOVIC}
\frac{m-1}\pi\int_{\mathbf{U}}  |f(z)|^{mp}\, (1-|z|^2)^{m-2}\, dA(z)\le
\|f\|^{mp}_p.
\end{equation}
Extremal      functions do not depend on the letter $m$. Explicitly, a function $f\in
H^p$ is extremal   for~\eqref{INEQ.BURBEA.MATELJEVIC.PAVLOVIC} if and only  if it has
the form
\begin{equation*}
f(z) = \frac  \lambda {(1-z\overline{w})^{2/p}},\quad z\in\mathbf{U}
\end{equation*}
for                               some $w\in\mathbf{U}$ and a constant $\lambda$. The
inequality~\eqref{INEQ.BURBEA.MATELJEVIC.PAVLOVIC}         may be found in the Burbea
paper~\cite{BURBEA.ILLINOIS.87}, but it  was proved earlier     by Mateljevi{\'c} and
Pavlovi{\'c}~\cite{MATELJEVIC.PAVLOVIC.VESNIK}        in a          different     way
-- considering one of the  Rudin problems for analytic functions in the unit polydisc.

If we take    $m=2$ in~\eqref{INEQ.BURBEA.MATELJEVIC.PAVLOVIC},    we have the modern
version of the Carleman inequality~\cite{CARLEMAN.MATH.Z}:   If $f(z)$ belongs to the
Hardy    space $H^p$, where $p$ is any positive  number, then it also belongs  to the
Bergman space $L^{2p}_a$,  and
\begin{equation}\label{INEQ.CARLEMAN.HP}
4\pi\int_{\mathbf{U}} |f(z)|^{2p} dA(z)\le
\left\{\int_{\mathbf{T}} |f(\zeta)|^p  |d\zeta|\right\}^{2}.
\end{equation}
A proof of~\eqref{INEQ.CARLEMAN.HP} and the corresponding equality statement  is also
exposed  in~\cite{VUKOTIC.MONTHLY} with the    observation that the original Carleman
approach                         leads to a simpler  proof of the result of Hardy and
Littlewood~\cite{HARDY.LITTLEWOOD.MATH.Z}        on boundedness      of the inclusion
operator from  $H^p$ into  $L^{2p}_a$.   A similar approach was given      earlier by
Mateljevi\'{c}~\cite{MATELJEVIC.LECTURE}.   The inequality may be also found  in  the
Strebel                            book~\cite[Theorem 19.9, pp. 96--98]{STREBEL.BOOK}.

The inequality~\eqref{INEQ.CARLEMAN.HP}     is related to the classical isoperimetric
inequality:         If $D$ is  a    simply--connected domain  in the  plain such that
$\partial D$ is   a rectifiable     curve, then the area  of  $D$  and the length  of
$\partial D$  satisfy
\begin{equation}\label{INEQ.ISOPERIMETRIC}
4\pi  \mathrm{Area}(D)\le  {\mathrm{Length}(\partial D)^2}
\end{equation}
with  equality if    and only if $D$ is a disc.            The analytic  proof of the
isoperimetric inequality is exposed in~\cite{VUKOTIC.MONTHLY},                    see
also~\cite{MATELJEVIC.FILOMAT.1},                        ~\cite{MATELJEVIC.FILOMAT.2},
~\cite{MATELJEVIC.PAVLOVIC.JMAA}. For      a discussion on~\eqref{INEQ.ISOPERIMETRIC},
various connections  with some known analytic                inequalities  (including
Carleman's one), we refer     to the survey article~\cite{OSSERMAN.BAMS} of  Osserman.

In~\cite{CARLEMAN.MATH.Z}  Carleman   deduced~\eqref{INEQ.ISOPERIMETRIC} for  minimal
surfaces.  In that case the relation~\eqref{INEQ.ISOPERIMETRIC}             should be
understood in             the  context    of  the intrinsic geometry    of a  surface.
Historically, that  was the first analytical  proof  of   the classical isoperimetric
inequality for  surfaces. Carleman's   original                 proof is based on the
Weierstrass--Enneper parametrization of minimal surfaces and the following inequality
which is also contained in  Theorem~\ref{TH.BURBEA.MARKOVIC}. Let $f_1,\, f_2\in H^1$,
then
\begin{equation*}\label{EQ.CARLEMAN.DOUBLE}
4\pi\int_{\mathbf{U}} |f_1(z)| |f_2(z)| dA(z)
\le\int_{\mathbf{T}} |f_1(\zeta)| |d\zeta|\int_{\mathbf{T}} |f_2(\zeta)|  |d\zeta|.
\end{equation*}

It is well known that the   classical isoperimetric inequality             holds for
simply--connected domains   on a surface of          the     non--positive  Gaussian
curvature.  One           may         use  a                          version of the
inequality~\eqref{INEQ.CARLEMAN.HP} for logarithmically     subharmonic functions in
order                  to derive the isoperimetric                inequality in this
circumference~\cite{BECKENBACH.TAMS.MINIMAL},       ~\cite{BECKENBACH.TAMS.NEGATIVE},
~\cite{LOZINSKY.WORK},~\cite{RADO.BOOK}.

\end{document}